\documentclass[11pt]{amsart}

\def\frk{\frak}               

\def\Phi{{\frk n}}
\def\Phi{{\frk N}}
\def\MN{{\mathcal N}}

\def\MC{{\mathcal C}}

\def\opn#1#2{\def#1{\operatorname{#2}}} 
\opn\chara{char} \opn\length{\ell} \opn\pd{pd} \opn\rk{rk}
\opn\projdim{proj\,dim} \opn\injdim{inj\,dim} \opn\rank{rank}
\opn\depth{depth} \opn\grade{grade} \opn\height{height}
\opn\embdim{emb\,dim} \opn\codim{codim}

\opn\Tr{Tr} \opn\bigrank{big\,rank}
\opn\superheight{superheight}\opn\lcm{lcm}
\opn\trdeg{tr\,deg}
\opn\reg{reg} \opn\lreg{lreg} \opn\ini{in} \opn\lpd{lpd}
\opn\size{size}\opn\bigsize{bigsize}
\opn\cosize{cosize}\opn\bigcosize{bigcosize}
\opn\sdepth{sdepth}\opn\sreg{sreg}
\opn\link{link}\opn\fdepth{fdepth}\opn\rev{rev}
\opn\Simp{Simp}\opn\Chord{Chord}

\opn\div{div} \opn\Div{Div} \opn\cl{cl} \opn\Cl{Cl}
\let\epsilon\varepsilon
\let\phi=\varphi
\let\kappa=\varkappa
\opn\Spec{Spec} \opn\Supp{Supp} \opn\supp{supp} \opn\Sing{Sing}
\opn\Ass{Ass} \opn\Min{Min}\opn\Mon{Mon} \opn\dstab{dstab} \opn\astab{astab}
\opn\Syz{Syz}
\opn\Ann{Ann} \opn\Rad{Rad} \opn\Soc{Soc}
\opn\Im{Im} \opn\Ker{Ker} \opn\Coker{Coker} \opn\Am{Am}
\opn\Hom{Hom} \opn\Tor{Tor} \opn\Ext{Ext} \opn\End{End}
\opn\Aut{Aut} \opn\id{id}

\opn\nat{nat}
\opn\pff{pf}
\opn\Pf{Pf} \opn\GL{GL} \opn\SL{SL} \opn\mod{mod} \opn\ord{ord}
\opn\Gin{Gin} \opn\Hilb{Hilb}\opn\sort{sort}
\opn\initial{init}
\opn\ende{end}
\opn\height{height}
\opn\type{type}
\opn\aff{aff} \opn\con{conv} \opn\relint{relint} \opn\st{st}
\opn\lk{lk} \opn\cn{cn} \opn\core{core} \opn\vol{vol}
\opn\link{link} \opn\star{star}\opn\lex{lex}
\opn\gr{gr}

\def\pot#1#2{#1[\kern-0.28ex[#2]\kern-0.28ex]}

\opn\dirlim{\underrightarrow{\lim}}
\opn\inivlim{\underleftarrow{\lim}}
\let\union=\cup
\let\sect=\cap

\let\Union=\bigcup

\def\Implies{\ifmmode\Longrightarrow \else
        \unskip${}\Longrightarrow{}$\ignorespaces\fi}
\def\implies{\ifmmode\Rightarrow \else
        \unskip${}\Rightarrow{}$\ignorespaces\fi}
\def\iff{\ifmmode\Longleftrightarrow \else
        \unskip${}\Longleftrightarrow{}$\ignorespaces\fi}

\let\:=\colon
\usepackage{amssymb, amscd, latexsym}
\usepackage{amsmath}
\usepackage{amsthm}
\usepackage{graphicx}
\usepackage[all]{xy}
\usepackage{float}
\usepackage{mathrsfs}
\usepackage{pgf,tikz}
\usetikzlibrary{arrows}
\usepackage{array}
\usepackage{arydshln}
\usepackage{hyperref}

\hypersetup{
pdftitle={Research Paper},
pdfauthor={Ali Akbar Yazdan Pour},
pdfsubject={Betti Number and chordality},
pdfcreator={Ali Akbar Yazdan Pour},
colorlinks,
linkcolor=blue,
citecolor=magenta,
anchorcolor=red,
bookmarksopen,
urlcolor=red,
filecolor=red,
pdfpagetransition={Wipe}
}

\theoremstyle{plain}
\newtheorem{thm}{Theorem}[section]
\newtheorem{cor}[thm]{Corollary}
\newtheorem{lem}[thm]{Lemma}
\newtheorem{prop}[thm]{Proposition}

\theoremstyle{definition}
\newtheorem{defn}[thm]{Definition}
\newtheorem{ex}[thm]{Example}

\theoremstyle{remark}
\newtheorem{rem}{Remark}

\def\cocoa{{\hbox{\rm C\kern-.13em o\kern-.07em C\kern-.13em o\kern-.15em A}}}

\def\Ker{{\rm Ker}}
\def\reg{{\rm reg}}
\def\rk{{\rm rank}\,}
\def\projdim{{\rm projdim}}

\def\C{{\mathcal C}}

\def\implies{\ifmmode\Rightarrow \else
        \unskip${}\Rightarrow{}$\ignorespaces\fi}
\def\Tor{\mathrm{Tor}}

\begin{document}

\title{simplicial orders and chordality}
\author[M. Bigdeli, J. Herzog, A. A. {Yazdan Pour} and  R. Zaare-Nahandi]{Mina Bigdeli, J\"urgen Herzog, {Ali Akbar} {Yazdan Pour} and  Rashid Zaare-Nahandi}

\address{Mina Bigdeli, department  of mathematics,  institute for advanced studies in basic sciences (IASBS),
	45195-1159 Zanjan, Iran} \email{mina.bigdeli@yahoo.com}

\address{J\"urgen Herzog, fachbereich mathematik, universit\"at duisburg-essen, fakult\"at f\"ur mathematik, 45117
	Essen, Germany} \email{juergen.herzog@uni-essen.de}

\address{Ali Akbar Yazdan Pour, department  of mathematics,  institute for advanced studies in basic sciences (IASBS),
		45195-1159 Zanjan, Iran} \email{yazdan@iasbs.ac.ir}

\address{Rashid Zaare-nahandi, department  of mathematics,  institute for advanced studies in basic sciences (IASBS),
	45195-1159 Zanjan, Iran} \email{rashidzn@iasbs.ac.ir}


\subjclass[2010]{Primary 13D02, 13P20; Secondary 05E45, 05C65.}
\keywords{chordal clutter, simplicial order, Betti number, $\lambda$-sequence, Hilbert function}

\begin{abstract}
Chordal clutters in  the sense of \cite{MNYZ} and \cite{BYZ} are defined via simplicial orders. Their circuit ideal has a linear resolution, independent of the characteristic of the base field. We show that any Betti sequence of an ideal with linear resolution appears as the Betti sequence of the circuit ideal of such a chordal clutter. Associated with any simplicial order is a sequence of integers which we call the $\lambda$-sequence of the chordal clutter. All possible $\lambda$-sequences are characterized. They are intimately related to the Hilbert function of a suitable standard graded $K$-algebra attached to the chordal clutter. By the  $\lambda$-sequence of a chordal clutter  we determine  other numerical invariants of the circuit ideal, such as the $\textbf{h}$-vector  and the  Betti numbers.
\end{abstract}
\maketitle

\section*{Introduction}

One of the most challenging problems in combinatorial commutative algebra is to give in combinatorial terms a characterization of those monomial ideals which have a linear resolution independent of the characteristic of the base field. In trying to solve this problem one may,  by using polarization,  focus on squarefree monomial ideals. Squarefree monomial ideals generated in degree 2 may be interpreted as edge ideals of graphs. In this particular case, the famous theorem of Fr\"oberg \cite{Fr} gives a complete answer to the above problem. Fr\"oberg's theorem says that the edge ideal $I(G)$ of a graph $G$ has a linear resolution if and only if the complementary graph $\bar{G}$ of  $G$ is chordal. In particular, for the edge ideal of a graph it does not depend on the characteristic  of the base field whether or not it has a linear resolution. One would like to have a similar theorem for squarefree monomial ideals generated in degree $d$, where $d$ is any integer $\geq 2$.

The $d$-uniform clutters generalize simple graphs,  which are just the $2$-uniform clutters. A $d$-uniform clutter $\MC$ on the vertex set $[n]=\{1,2,\ldots,n\}$ is nothing but a collection of $d$-subsets of $[n]$. The complementary clutter  $\bar{\MC}$ of $\MC$ is the set of all $d$-element subsets of $[n]$ which do not belong to $\MC$.  The circuit ideal of $\MC$ is the ideal generated by  the monomials $\textbf{x}_F$ with $F\in \bar{\MC}$. The analogue of Fr\"oberg's would then say that the circuit ideal of $\MC$ has a linear resolution, independent of the characteristic of the base field, if and only if $\MC$ is chordal. But what does it mean that $\MC$ is chordal? There have been several other attempts to define chordal clutters, see \cite{Nevo,ConnonFaridi,Emtander,Woodroofe}.  Here we consider the  concept of chordality of clutters  as it was introduced in \cite{MNYZ}. It was shown in \cite{BYZ} that all previously defined chordal clutters are  chordal in this new sense, and that the circuit ideal of a chordal clutter has a linear resolution. However, the converse is still an open question. The definition of chordality  given in \cite{MNYZ} imitates Dirac's characterization of chordal graphs.  It is required that the clutter admits a  simplicial order. Roughly speaking, a {\em simplicial order} for a $d$-uniform clutter $\MC$ is a sequence  $\textbf{e}= e_1,\ldots,e_r$ of $(d-1)$-sets such that for each $i$,  the set  $e_i$ has a clique as its  closed neighborhood in $\MC_{e_1\dots e_{i-1}}$, where $\MC_{e_1\dots e_{i-1}}$ is obtained from $\MC$ by `deleting' $e_1,\ldots,e_{i-1}$. Furthermore, it is required that $\MC_{e_1,\ldots e_r}=\emptyset$. The precise definition is given in Section~\ref{pre}.

The cardinality of the open neighborhood  belonging to $e_i$, that is to say, the cardinality of the set  of all elements $c\in [n]$ for which $e\union\{c\}\in\MC$,  is denoted by $N_i$. Thus, with each simplicial order $\textbf{e}$  of a chordal clutter there is  associated the multiset $\{N_1,\ldots,N_r\}$. Simple examples show that a chordal clutter may have different simplicial orders. However, as shown in Corollary~\ref{remarkable},  all simplicial orders have the same multisets. Therefore we can talk of {\em the} multiset of a chordal clutter. The invariance of the multiset is a consequence of Proposition~\ref{fvector},   where it is shown that the $\textbf{f}$-vector of the clique complex $\Delta(\MC)$ of a chordal clutter $\MC$ is determined by the multiset of any of its simplicial orders. From this result it can be also derived that the multiset of $\MC$ determines the $\textbf{h}$-vector  of $\Delta(\MC)$ as well as the Betti sequence of the circuit ideal of $\MC$, see Corollary~\ref{h-vector} and Corollary~\ref{bettinumbers}. More important is the result, stated in Theorem~\ref{all}, which says that the Betti sequence of any ideal with linear resolution is the Betti sequence of the circuit ideal of a chordal clutter. So numerically, the circuit ideal of chordal clutters coincide with the ideals with linear resolution. This fact supports our expectation that the circuit ideal of  the chordal clutters are precisely the monomial ideals which admit a linear resolution, independent of the characteristic of the base field.

Since each chordal clutter determines a multiset $\MN=\{N_1,\ldots,N_r\}$, and since this multiset determines all the  algebraic and homological data of the clutter,  it is of interest to know which finite multisets of positive integers occur as the multiset of a chordal clutter. With the multiset $\MN$ we associate a sequence $\lambda(\MN)=\lambda_1(\MN),\lambda_2(\MN),\ldots$, where $\lambda_i(\MN)$ is the number of elements $N_j$  with $N_j=i$. It is obvious that any finite multiset of positive integers is in bijection to a sequence of integers $\lambda_1,\lambda_2,\ldots$, which is eventually constant zero. We denote by $\lambda(\MC)$ the $\lambda$-sequence of the multiset associated with $\MC$. In the last section of this paper the question is studied which integer sequences are the $\lambda$-sequences of  a chordal clutter. The main result is formulated in Theorem~\ref{possiblelambda}, where it is shown that the $\lambda$-sequence of a chordal $d$-clutter on the vertex set $[n]$ is intimately related to the Hilbert function of a suitable $K$-algebra $R$ of embedding dimension $\leq d$ with $\dim_KR_{i}=0$ for $i\geq n-d+1$. By using this characterization a strict upper bound of $\lambda_i(\MC)$ for each $1\leq i\leq n-d$ is given in Corollary~\ref{bound}. Finally in Proposition~\ref{complete} we give a precise formula for the $\lambda$-sequence of the complete $d$-clutter on $[n]$. Indeed it is shown that $\lambda_{i}(\MC_{n,d})={n-i-1\choose d-2}$.

In \cite{code}, a computer program has been provided in order to check chordality of a given $d$-uniform clutter. A modified version of this program, available at \cite{code2}, provides some numerical data about circuit ideal of chordal clutters. More precisely, this program first checks whether a given $d$-uniform clutter $\mathcal{C}$ is chordal. Then, it computes the simplicial multiset of $\mathcal{C}$, its $\lambda$-sequence, the $\textbf{h}$-vector of $\Delta = \Delta \left( \mathcal{C} \right)$ and the Betti sequence of the ideal $I_\Delta$.

\section{Preliminaries}
\label{pre}

In this section we fix notation and recall some concepts and results which will be used in this paper.

\subsection{Simplicial complexes}
A \textit{simplicial complex} $\Delta$  over a set of vertices $V=\{v_{1}, \ldots, v_{n} \}$, is a collection of subsets of $V$, with the property that:
\begin{itemize}
\item[(a)] $\{ v_{i} \} \in \Delta $ for all $i$;
\item[(b)] if $F\in \Delta$, then all subsets of $F$ are also in $\Delta$ (including the empty set).
\end{itemize}

The elements $F\in \Delta$ are called the {\em faces} of $\Delta$ and \textit{dimension} of each face $F \in \Delta$ is $\dim F =|F|-1$. Also \textit{dimension} of $\Delta$, $\dim \Delta$, is given by $\max\{\dim F\:\, F\in \Delta\}$.

\begin{defn}[\textbf{f}-vector]
Let $\Delta$ be a simplicial complex of dimension $\delta-1$ on the vertex set $V$ and  $f_i = f_i(\Delta) $ denote the number of faces of $\Delta$ of dimension $i$. Thus, in particular $f_{-1} =1, f_0 = n$. The sequence
\begin{equation*}
\textbf{f}(\Delta) = \left( f_{-1}, f_0, \ldots, f_{\delta-1} \right)
\end{equation*}
is called the $\textbf{f}$-\textit{vector} of $\Delta$. Also, the polynomial
\begin{equation*}
\textbf{f}_\Delta(t) = f_{-1} + f_0 t + \cdots + f_{\delta-1} t^{\delta} \in \mathbb{Z}[t]
\end{equation*}
is called the $\textbf{f}$-\textit{polynomial} of $\Delta$.
\end{defn}

The algebraic object attached to  a simplicial complex is the so-called Stanley--Reisner ring.

\begin{defn}
Let $\Delta$ be a simplicial complex on the vertex set $[n]=\{1,2,\ldots,n\}$, and  let $K$ be a field. The {\em Stanley--Reisner ideal} $I_\Delta$ of $\Delta$ is the squarefree monomial ideal in the polynomial ring $S=K[x_1,\ldots,x_n]$ generated by the monomials $\textbf{x}_F$ with $F\subset [n]$ and $F\notin \Delta$. Here, for $F=\{i_1,\ldots,i_r\}$, we have set $\textbf{x}_F=\prod_{j=1}^rx_{i_j}$.
The $K$-algebra $K[\Delta]=S/I_\Delta$  is called the {\em Stanley--Reisner ring} of $\Delta$.
\end{defn}

Note that $K[\Delta]$ is naturally standard graded. Thus, its Hilbert series $\Hilb(K[\Delta], t)$ is a rational function of the form $Q(t)/(1-t)^\delta$, where $\delta=\dim K[\Delta]$ and $Q(t)=h_0+h_1t+\cdots+ h_\delta t^\delta$ is a polynomial of degree at most $\delta$ with integer coefficients. The vector $\textbf{h} (\Delta) =(h_0,\ldots,h_\delta)$ is called the $\textbf{h}$-\textit{vector} of $\Delta$.

\begin{prop}[{\cite[Lemma 5.1.8]{BH}}] \label{Formula for h-vector}
Let $\Delta$ be a simplicial complex of dimension $\delta-1$ and $\text{{\rm \textbf{h}}}=(h_i)$ {\rm(}res. $\text{{\rm \textbf{f}}}=(f_i)${\rm )} be the $\text{{\rm \textbf{h}}}$-vector  {\rm(}res. $\text{{\rm \textbf{f}}}$-vector{\rm )} of $\Delta$. Then
\begin{equation*}
\sum\limits_{i=0}^{\delta} h_i t^i = \sum\limits_{i=0}^{\delta} f_{i-1} t^i (1-t)^{\delta-i} \quad \text{and}\quad \sum_{i=0}^{\delta} h_i t^i(1+t)^{\delta-i} = \sum\limits_{i=0}^{\delta} f_{i-1} t^i.
\end{equation*}
In particular, for $k > \delta$, $h_k =0$ and for $0 \leq k \leq \delta$,
\begin{itemize}
\item[\rm (i)] $h_k =\sum\limits_{i=0}^{k} (-1)^{k-i} { \delta-i \choose k-i }f_{i-1}$;
\item[\rm (ii)] $f_{k-1} = \sum\limits_{i=0}^{k} {{\delta-i} \choose {k-i}} h_i$.
\end{itemize}
\end{prop}

\subsection{Clutters}
We recall the concept of clutters and circuit ideals as well as the concept of simplicial orders which will lead to a definition of  chordality.

\begin{defn}[Clutter] \label{SC}
A \textit{clutter} $\C$ on the vertex set $[n]$ is a collection of subsets of $[n]$, called \textit{circuits} of $\C$, such that if $F_1$ and $F_2$ are distinct circuits, then $F_1 \nsubseteq F_2$.
A \textit{$d$-circuit} is a circuit consisting of exactly $d$ vertices, and a clutter is called \textit{$d$-uniform}, if every circuit has $d$ vertices.
A $(d-1)$-subset $e \subset [n]$ is called a \textit{submaximal circuit} of $\C$, if there exists $F \in \C$ such that $e \subset F$. The set of all submaximal circuits of $\C$ is denoted by ${\rm SC}(\C)$.
\end{defn}

For a non-empty clutter $\C$ with the vertex set $[n]$, we define the ideal $I \left( \C \right)$, as follows:
$$I(\C) = \left(  \textbf{x}_F \colon \quad F \in \C \right),$$
and we set $I(\emptyset) = 0$.

Let $n$, $d$ be positive integers. For $n \geq d$, we define $\C_{n,d}$, the \textit{complete $d$-uniform clutter} on $[n]$, as follows:
$$\C_{n,d} = \left\{ F \subset [n] \colon \quad |F|=d \right\}.$$
In the case that $n<d$, we let $\C_{n,d}$ be some isolated points. It is well-known that, for $n \geq d$ the ideal $I \left( \C_{n,d} \right)$ has a $d$-linear resolution (see e.g. \cite[Example 2.12]{MNYZ}).

If $\C$ is a $d$-uniform clutter on $[n]$, we define $\bar{\C}$, the \textit{complement} of $\C$, to be
\begin{equation*}
\bar{\C} = \C_{n,d} \setminus \C = \{F \subset [n] \colon \quad |F|=d, \,F \notin \C\}.
\end{equation*}

Frequently in this paper, we take a $d$-uniform clutter $\C \neq \C_{n,d}$ with the vertex set $[n]$ and consider the squarefree monomial ideal $I(\bar{\C})$ in the polynomial ring $S=K[x_1, \ldots, x_n]$. The ideal $I \left( \bar{\C} \right)$ is called the \textit{circuit ideal} of $\C$.

\begin{defn}
Let $\C$ be a $d$-uniform clutter on $[n]$. A subset $V \subset [n]$ is called a \textit{clique} in $\C$, if all $d$-subsets of $V$ belong to $\C$. Note that a subset of $[n]$ with less than $d$ elements is supposed to be a clique.
\end{defn}

The set of cliques of $\MC$ forms a simplicial complex, denoted $\Delta(\MC)$, which is called the {\em clique complex} of $\MC$. Its Stanley--Reisner ideal $I_{\Delta(\MC)}$ coincides with the circuit ideal $I(\bar{\MC})$ of $\MC$ \cite[Proposition 4.4]{mar2}.

\medskip
For any $(d-1)$-subset $e$ of $[n]$, the set
\begin{equation*}
\mathrm{N}_{\C} \left( e \right) = \lbrace c \in \left[ n \right] \colon \quad e \cup \lbrace c \rbrace \in \mathcal{C} \rbrace
\end{equation*}
is called the \textit{open neighborhood} of $e$ in $\mathcal{C}$, while the \textit{close neighborhood} of $e$ in $\mathcal{C}$ is defined to be $\mathrm{N}_\C\left[ e \right] = e \cup \mathrm{N}_{\C} \left( e \right)$. We say that $e \in \mathrm{SC}(\C)$ is \textit{simplicial} in $\C$, if $\mathrm{N}_{\C} \left[ e \right]$ is a clique in $\mathcal{C}$. Let us denote by $\mathrm{Simp} \left( \mathcal{C} \right)$, the set of all simplicial elements of $\mathcal{C}$.

\begin{defn}
Let $\mathcal{C}$ be a $d$-uniform clutter and let $e$ be a $(d-1)$-subset of $[n]$.  By $\mathcal{C} \setminus e$ we mean the $d$-uniform clutter
$$ \left\{ F \colon \quad F \in \mathcal{C}, \ e \not\subset F \right\}.$$
It is called the \textit{deletion} of $e$ from $\mathcal{C}$. In the case that $e$ is not a submaximal circuit of $\C$, we have $\mathcal{C} \setminus e = \C$.
\end{defn}

Now we come to the crucial definition of this paper.

\begin{defn}\rm
Let $\mathcal{C}$ be a $d$-uniform clutter. We call $\mathcal{C}$ a \textit{chordal clutter}, if either $\mathcal{C}=\emptyset$, or $\mathcal{C}$ admits a simplicial submaximal circuit $e$ such that $\mathcal{C}\setminus e$ is chordal.
\end{defn}
Following the notation in \cite{MNYZ}, we use $\mathfrak{C}_d$, to denote the class of all $d$-uniform chordal clutters.

\begin{defn}
A sequence of submaximal circuits of $\mathcal{C}$, say $e_1, \ldots, e_t$, is called a \textit{simplicial sequence} in $\mathcal{C}$ if $e_1$ is  simplicial  in $\mathcal{C}$ and $e_i$ is simplicial in $\left( \left( \left( \mathcal{C} \setminus e_{1} \right) \setminus e_2 \right)  \setminus \cdots \right) \setminus e_{i-1}$ for all $i>1$,
\end{defn}

The definition of chordal clutters can be restated as follows: the $d$-uniform clutter $\mathcal{C}$ is called chordal if either $\mathcal{C} = \emptyset$, or else there exists a simplicial sequence in $\mathcal{C}$, say $e_1, \ldots, e_r$, such that $\left( \left( \left(  \mathcal{C} \setminus e_{1} \right) \setminus e_2 \right) \setminus \cdots \right) \setminus e_{r} = \emptyset$.
To simplify the notation, we use $\mathcal{C}_{e_1 \dots e_i}$  for $\left( \left( \left( \mathcal{C} \setminus e_{1} \right) \setminus e_2 \right) \setminus \cdots \right) \setminus e_{i}$.

\medskip
Let $\MC$ be a $2$-uniform clutter. Then $\MC$ is nothing but a finite simple graph $G$.   By the theorem of Dirac \cite{Dirac},  the $2$-uniform clutter $\MC$ is chordal in our sense if and only if $G$ is a chordal graph. Thus, the definition of chordal graphs given here, is one of the possible natural  extensions of chordal graphs to chordal clutters.

In \cite[Remark 3.10]{MNYZ} it is shown that in analogy to Fr\"oberg's theorem,  the circuit ideal $I(\bar{\MC})$  of a $d$-uniform clutter $\C$ has a linear resolution over any field, if $\MC$ is chordal. The converse of this statement however, which is true for $2$-uniform clutters (graphs), is not known for general $d$-uniform clutters.

\medskip
In combinatorics it is common to consider multisets. These are sequences of elements whose order are disregarded. Notation is as for sets; for example, $\{1,2,2,3\}$ is a multiset. Elements may be repeated, but the order does not matter. For instance, the multisets $\{1,2,2,3\}$ and $\{3,2,1,2\}$ coincide.

\begin{defn}
Let $\C$ be a chordal clutter with a simplicial sequence $\textbf{e}= e_1, \ldots, e_r$ such that $\mathcal{C}_{e_1, \ldots, e_r} = \emptyset$, and let $N_1:=  \lvert \mathrm{N}_\C \left(e_1 \right)\rvert$ and $N_i := \lvert \mathrm{N}_{\C_{e_1 \cdots e_{i-1}}} \left( e_i \right) \rvert$, for $i>1$. Then the sequence $\textbf{e}$ is called the \textit{simplicial order} on $\C$ and the multiset $\left\{ N_1, \ldots, N_{r} \right\}$ is called the  \textit{simplicial multiset} of $\C$ associated with  $\textbf{e}$.
\end{defn}

\section{The $\textbf{f}$-vector and $\textbf{h}$-vector of a chordal clutter}

In this section we determine the $\textbf{f}$- and $\textbf{h}$-vector of a chordal clutter $\MC$ and show  that the multiset of a simplicial order of $\MC$ does not depend on the particular  given simplicial order.

\begin{prop}[$\textbf{f}$-vector of chordal clutters]
\label{fvector}
Let $\C$ be a $d$-uniform chordal clutter on the vertex the set $[n]$ with a simplicial order $\textbf{e}=e_1, \ldots, e_r$. Let $\left\{ N_1, \ldots, N_r \right\}$ be the simplicial multiset of $\textbf{e}$, $M_i = \sum_{k=1}^r {N_k \choose i}$ and $\Delta = \Delta \left( \C \right)$ be the clique complex of $\C$. Then
\begin{align}
\textbf{\rm{\textbf{f}}}_\Delta (t)& = \sum\limits_{i=0}^{d-1} {n \choose i} t^{i} + t^{d-1} \sum\limits_{i=1}^{r} M_i t^i \label{f-vector, No1}\\
& = \sum\limits_{i=0}^{d-1} {n \choose i}t^{i} + t^{d-1} \sum\limits_{i=1}^{r} \left( \left( 1 +t\right)^{N_i} -1 \right). \label{f-vector, No2}
\end{align}
In particular, $\dim \Delta = \max \left\{N_1, \ldots, N_r \right\} + (d-2)$.
\end{prop}

\begin{proof}
First note that, $f_{i-1} = {n \choose {i}}$, for $0 \leq i \leq d-1$. In order to obtain (\ref{f-vector, No1}), let $\C'=\C \setminus e_1$ and $\Delta'= \Delta\left( \C' \right)$. Since $e_1$ is a simplicial submaximal circuit of $\C$, we may easily verify that
\begin{equation} \label{f-vector No3}
\textbf{f}_\Delta(t) - \textbf{f}_{\Delta'}(t) = {{N_1} \choose 1}t^{d} + \cdots + {{N_1} \choose {N_1}}t^{( d-1) + N_1}.
\end{equation}
Since $\C'$ is again chordal with the simplicial order $e_2, \ldots, e_r$, we may use induction (on the number of circuits of $\C$), to obtain
$$\textbf{f}_{\Delta'}(t) =  \sum\limits_{i=0}^{d-1} {n \choose {i}} t^i +t^{d-1} \sum\limits_{i=1}^{r} M'_i t^{i},$$
where $M'_i = \sum_{k=2}^r {N_k \choose i}$. Now, equation~(\ref{f-vector No3}) yields (\ref{f-vector, No1}). The equality~(\ref{f-vector, No2}) is now obvious, by comparing the coefficients of two polynomials.
\end{proof}

 Note that a chordal clutter may have different simplicial orders. A simple example where this phenomenon appears   is the chordal clutter $\MC=\{123,124,134,234,145\}$ which has the simplicial orders $\textbf{e}= 12,23,13,14$ and $\textbf{e}'=15,14,12,23$.   Here and in the sequel it is convenient to denote the set $\{i_1,\ldots,i_k\}$ by $i_1i_2\ldots i_k$. The simplicial multiset for $\textbf{e}$ is $\{2,1,1,1\}$ and that of $\textbf{e}'$ is $\{1,2,1,1\}$. Thus  the simplicial multisets of $\textbf{e}$ and $\textbf{e}'$ coincide,  which happens to be the case  not by accident. Indeed,  Proposition~\ref{fvector} yields the following remarkable result.

 \begin{cor}
 \label{remarkable}
 Let $\MC$ be a chordal clutter with simplicial orders $\textbf{\rm \textbf{e}}$ and $\textbf{\rm \textbf{e}}'$. Then the simplicial multisets associated with  $\textbf{\rm \textbf{e}}$ and $\textbf{\rm \textbf{e}}'$ coincide. In particular, all simplicial orders of $\MC$  have the same length.
 \end{cor}

\begin{proof}
Let $\{N_1,\ldots,N_r\}$ be the simplicial multiset associated with  $\textbf{e}$, and $\{N_1',\ldots,N_s'\}$ be the simplicial multiset associated with  $\textbf{e}'$. Then (\ref{f-vector, No2}) implies that
\[
\sum_{i=1}^{r} \left( \left( 1+t \right)^{N_i} -1 \right)=\sum_{i=1}^{s} \left( \left( 1+t \right)^{N_i'} -1 \right)
\]
The substitution $t\mapsto t-1$ then yields the identity
\[
\sum_{i=1}^{r} \left(t^{N_i} -1 \right)=\sum_{i=1}^{s} \left(t^{N_i'} -1 \right).
\]
Substituting then $t$ by $0$, we see that $r=s$, and consequently $\sum_{i=1}^{r} t^{N_i}=\sum_{i=1}^rt^{N_i'}$. Note that $\sum_{i=1}^{r} t^{N_i}=\sum_ja_jt^{N_j}$, where $a_j=|\{i\:\; N_i=j\}|$, and similarly $\sum_{i=1}^rt^{N_i'}=\sum_jb_jt^{N'_j}$, where  $b_j=|\{i\:\; N'_i=j\}|$. Since $a_j=b_j$ for all $j$, it follows that the simplicial multisets associated with  $\textbf{e}$ and $\textbf{e}'$ are equal, as desired.
\end{proof}

Since the multisets of any two simplicial orders of a $d$-uniform chordal clutter $\MC$ coincide, we call this common multiset, the {\em multiset} of $\MC$. Let $\{N_1,\ldots,N_r\}$ be the multiset of $\MC$. Then, as we have seen in Proposition~\ref{fvector}, the dimension of $\Delta(\MC)$ is equal to $N+(d-2)$, where $N=\max\{N_1,\ldots,N_r\}$.

\begin{cor}
\label{h-vector}
Let $\MC$ be a $d$-uniform chordal clutter with multiset $\{N_1,\ldots,N_r\}$ and $\Delta = \Delta (\C)$ be the clique complex of $\C$. Then
\[
\textbf{h}_\Delta (t) =\sum_{i\geq 0} h_it^i
=\sum_{i=0}^{d-1}{n\choose i}t^i(1-t)^{N+d-i-1}+t^{d-1}\sum_{i=1}^r\left((1-t)^{N-N_i}-(1-t)^{N}\right).
\]
\end{cor}

\begin{proof}
We use the identity $\textbf{f}_\Delta(t) = \sum_{i=0}^{\delta+1} h_i t^i(1+t)^{\delta+1-i}$  from Proposition~\ref{Formula for h-vector}, where $\delta =N+(d-2)$ is the dimension of the clique complex of $\MC$, and apply the substitution $t\mapsto t/(1-t)$ to obtain
\[
\sum_{i=0}^{\delta+1} h_i t^i=(1-t)^{\delta+1}\textbf{f}_\Delta(t/(1-t)).
\]
Evaluating the right hand side of the equation yields the desired conclusion.
\end{proof}

A precise formula for the $\textbf{h}$-vector, in terms of the multiset of the given chordal clutter can be directly obtained from Proposition~\ref{Formula for h-vector}(i). Indeed,
let $\C$ be a $d$-uniform chordal clutter on the vertex set $[n]$, with (unique) multiset $\mathcal{N} =\left\{ N_1, \ldots, N_r \right\}$. Let $\Delta = \Delta \left( \C\right)$ be its clique complex, $M_i = \sum_{k=1}^{r} {N_k \choose i}$ and $N=\mathop{\max}\limits_{1 \leq i \leq r} \{ N_i \}$. If $\textbf{h} = \left( h_i \right)$ is the $\textbf{h}$-vector of $\Delta$, then  $h_k=0$ for $k> N+(d-1)$,  and
\begin{align*}
h_k = \begin{cases}
\sum\limits_{i=0}^k \left( -1 \right)^{k-i} {{N+d-1-i} \choose {k-i}} {n \choose i}, & \text{if } 0 \leq k \leq d-1 \\
\sum\limits_{i=0}^{d-1} \left( -1 \right)^{k-i} {{N+d-1-i} \choose {k-i}} {n \choose i} + \sum\limits_{i=d}^{k} \left( -1 \right)^{k-i} {{N+d-1-i} \choose {k-i}} M_{i-d+1}, & \text{if } d \leq k \leq N+d-1.
\end{cases}
\end{align*}
In particular, if $\dim \Delta = d-1$, then $\mathcal{N} = \left\{ 1, \ldots, 1 \right\}$ and $h_i =0$, for all $i>d$. Moreover,
\begin{align} \label{h-vector of forest}
h_k = \begin{cases}
\sum\limits_{i=0}^k \left(-1 \right)^{k-i} {{d-i} \choose {k-i}} {n \choose i}, & \text{if } 0 \leq k \leq d-1 \\
r + \sum\limits_{i=0}^{d-1} \left(-1 \right)^{d-i} {n \choose i}, & \text{if } k=d.
\end{cases}
\end{align}
It is worth to note that $r = M_1 = | \C | = e \left( I_\Delta \right)$, where $e \left( I_\Delta \right)$ denotes the multiplicity of $I_\Delta$. We may use formula~(\ref{h-vector of forest}) in order to compute the Betti numbers of the circuit ideal of a forest (i.e. a graph with no cycle).

\section{The Betti numbers of the circuit ideal of a chordal clutter}

Let $I\subset S$ be a graded ideal with a $d$-linear resolution and $\dim S/I= \delta$.
By \cite[Lemma~4.1.13]{BH} it follows that
\begin{align}\label{betti}
1+ \sum\limits_{i=0}^{n-1} \left( -1 \right)^{i+1} \beta_i t^{i+d} = \left( 1-t \right)^{n-\delta} \sum\limits_{i=0}^{\delta+1} h_i t^i.
\end{align}
Here $\beta_i$ is the $i$th Betti number of $I$. In other words, if $\mathbb{F}$ is the graded minimal free $S$-resolution of $I$,  then $\beta_i$ is the rank of the free $S$-module $F_i$.

\medskip
Now we apply (\ref{betti}) to  the circuit ideal of a chordal clutter. Then by  Corollary~\ref{h-vector}, we obtain the following result.

\begin{cor}
\label{bettinumbers}
Let $\MC$ be a $d$-uniform chordal clutter with multiset $\{N_1,\ldots,N_r\}$ and set $\beta_i= \beta_i(I(\bar{\MC}))$ for all $i$. Then
\begin{align*}
1+\sum\limits_{i=0}^{n-1} \left( -1 \right)^{i+1} \beta_i t^{i+d}  = &\sum_{i=0}^{d-1}{n\choose i}t^i(1-t)^{n-i}\\
&+ t^{d-1}\sum_{i=1}^r\left((1-t)^{n-N_i -d + 1}-(1-t)^{n-d+1}\right).
\end{align*}


\end{cor}

As can be seen from this result, the Betti numbers $\beta_i(I(\bar{\MC}))$  can be expressed (in a complicated way) in terms of the multiset of $\MC$. In the following example, we use Corollary~\ref{bettinumbers}, to compute the Betti numbers of a forest. Recall that a forest, is a graph with no cycle. These graphs are obviously chordal, because every non-trivial subgraph of a forest has a vertex of degree $1$, which is obviously a simplicial submaximal circuit.

\begin{ex}[Betti numbers of the circuit ideal of a forest]
Let $G$ be a graph and $\Delta= \Delta(G)$ be its clique complex. One may check that the followings are equivalent:
\begin{itemize}
\item[\rm (a)] $G$ is chordal and $\dim \Delta = 1$;
\item[\rm (b)] $G$ is a forest.
\end{itemize}
Now, let $G$ be a forest on the vertex set $[n]$ and $I= I \left( \bar{G} \right)$ be the circuit ideal of $G$. Then, by (\ref{h-vector of forest}), we conclude that $\textbf{h}  = \left( 1, n-2, 0 \right)$ is the $\textbf{h}$-vector of $\Delta$. By formula~(\ref{betti}), (or directly by Corollary~\ref{bettinumbers}), we obtain
\begin{equation*}
\beta_i \left( I \right) = \left( i+1 \right) {n-1 \choose i+2}.
\end{equation*}
In particular, $\projdim (I) = n-3$.
\end{ex}

\medskip
A sequence $\beta=(\beta_0,\ldots,\beta_{n-1})$ of  integers is called the {\em Betti sequence} of a graded ideal, if there exists  a graded ideal $I\subset S=K[x_1,\ldots,x_n]$  such that $\beta_i=\beta_i(I)$, for all $i$.

The question arises whether all ideals which have a linear resolution occur as the circuit ideal of chordal clutters. In support of an affirmative answer to this question we have

\begin{thm}
\label{all}
Let $\beta=(\beta_0,\ldots,\beta_{n-1})$ be a sequence of integers. Then $\beta$ is a Betti sequence of a graded  ideal with linear resolution if and only if $\beta$  is the Betti sequence of the  circuit ideal of a chordal clutter.
\end{thm}

\begin{proof}
We need only to show that if $\beta$  is the Betti sequence of a graded  ideal $I$ with linear resolution, then there exists  a chordal clutter whose circuit ideal has the same Betti numbers. We may assume that $K$ is infinite. We denote by $\Gin(I)$ the generic initial ideal of $I$ with respect to the reverse lexicographic order induced by $x_1>x_2>\cdots >x_n$. It follows from  the theorem of Bayer and Stillman \cite{BS} (see also \cite[Corollary 4.3.18(d)]{HH}) that $\Gin(I)$ has again a linear resolution. Furthermore, since $S/I$ and $S/\Gin(I)$ have the same Hilbert series, we conclude that the Betti sequence of $I$ coincides with that of $\Gin(I)$. It has been noted in \cite[Lemma 1.4]{CHH}, as a consequence of a result of Eisenbud et al. \cite[Proposition 10]{ERT}, that $\Gin(I)$ is stable (independent of the characteristic of $K$). The Eliahou-Kervaire resolution of a stable ideal (see \cite{EK} or \cite[Corollary 7.2.3]{HH}) shows that the  Betti numbers of a stable ideal do not depend on the characteristic of $K$. Thus we may now assume that $\chara(K)=0$. Taking now the generic initial ideal, but this time of $\Gin(I)$,  we obtain by \cite[Proposition 4.2.6(a)]{HH} a strongly stable ideal which, by the similar  arguments as above, has a linear resolution with the same Betti sequence as that of $\Gin(I)$, and hence that of $I$.
The lemmata \cite[Lemma 11.2.5]{HH} and \cite[Lemma 11.2.6]{HH} say that if $I$ is a strongly stable ideal, then  the stretched ideal $I^\sigma$ is squarefree strongly stable and admits the same Betti sequence. Thus we see that the set  of Betti sequences of graded ideals with linear resolution coincides with the set of Betti sequences of  squarefree strongly stable ideals. Now we use the result, shown in \cite[Theorem 2.4]{XY}, that a clutter is chordal if its  circuit ideal is squarefree strongly stable ideal. This completes the proof.
\end{proof}

\section{The multiset of a chordal clutter}

By Corollary~\ref{bettinumbers} and Theorem~\ref{all}, the  multiset of a $d$-uniform chordal clutter parametrize  all possible Betti sequences of ideals with $d$-linear resolution. But what are the possible multisets of $d$-uniform chordal clutters?

For a given multiset $\MN=\{N_1,N_2,\ldots, N_r\}$ of positive integers we define the numbers
\[
\lambda_i(\MN)=|\{j\colon \; |N_j|=i\}| \quad\text{for}\quad  i=1,2,\ldots.
\]
Obviously, for any sequence of non-negative integers $\lambda_1,\lambda_2,\ldots,$ with $\lambda_i=0$ for $i\gg0$, there exists a unique multiset $\MN$ with $\lambda_i(\MN)=\lambda_i$ for all $i$. This shows that the set of multisets of positive integers is  in bijection to the set of sequences of non-negative integers which are eventually constant zero.

Let $\MN$ be the multiset of a chordal clutter $\MC$. For short, we set $\lambda_i\left( \mathcal{C} \right):=\lambda_i\left( \mathcal{N} \right)$ for all $i$. The sequence $\lambda_1 \left( \mathcal{C} \right), \lambda_2 \left( \mathcal{C} \right), \ldots$ is called the \textit{$\lambda$-sequence} of $\mathcal{C}$. Then, to characterize the possible multisets of chordal clutters is equivalent to characterize those sequences $\lambda_1,\lambda_2,\ldots $ of non-negative integers which are eventually constant zero and such that there exists a chordal clutter $\MC$ with $\lambda_i=\lambda_i(\MC)$, for all $i$.

\medskip
As shown in Proposition~\ref{fvector} and in the proof of Corollary~\ref{remarkable}, the multiset of a chordal clutter, and hence its $\lambda$-sequence,  is determined by the $\textbf{f}$-vector of the clique complex of the clutter, and vice versa. On the other hand, the $\textbf{f}$-vector and the $\textbf{h}$-vector of a simplicial complex determine each other as well. The same holds true for the $\textbf{h}$-vector of a simplicial complex $\Delta$ and the Betti sequence of  its Stanley--Reisner ideal $I_\Delta$, provided that $I_\Delta$ has a linear resolution. In conclusion we see that the $\lambda$-sequence of a chordal clutter and the Betti sequence of its circuit ideal determine each other, and consequently, due to Theorem~\ref{all},  all possible $\lambda$-sequences of chordal clutters appear as $\lambda$-sequences of clutters whose circuit ideal is squarefree strongly stable. We may therefore assume that $\MC$ is a $d$-uniform clutter whose circuit ideal  $I(\bar{\MC})$ is squarefree strongly stable. Thus,  in order to identify the possible $\lambda$-sequence of a chordal clutter we have to recall some facts about squarefree strongly stable ideals.

Let $S=K[x_1,\ldots,x_n]$ be the polynomial ring over the field $K$ in the indeterminates $x_1,\ldots,x_n$, and let $I\subset S$ be a squarefree strongly stable ideal generated in degree $d$. As usual,  $G(I)$ denotes the unique minimal set of monomial generators of $I$. We furthermore set $G_i(I)=\{u\in G(I)\:\; m(u)=i\}$, where $m(u)=\max\{j\:\; \text{$x_j$ divides $u$} \}$.  The cardinality of $G_i(I)$ will be denoted by  $m_i(I)$. These numbers are important invariants  for squarefree strongly stable ideals.  Let $I_{[j]}$ be the ideal generated by all squarefree monomials of degree $j$ in $I$. Then $I_{[j]}$ is again squarefree strongly stable. Finally, we denote by $\mu_j(I)$ the number of monomial generators of  $I_{[j]}$.

\begin{lem}
\label{generators}
Let $I\subset S$ be a squarefree strongly stable ideal generated in degree $d$. Then
\[
\mu_{d+j}(I)=\sum_{i=0}^{n-d}{n-d-i\choose j}m_{d+i}(I),
\]
for all $j\geq 0$.
\end{lem}

\begin{proof}
Let $S_{[i,j]}$ be the set of all squarefree monomials $u$ in $S$ of degree $j$ with the property that $\min(u)=\min\{k \colon \; x_j \text{ divides } u \} \geq i$.

We claim that
\begin{eqnarray}
\label{disjoint}
G(I_{[d+j]})= \Union_{i=0}^{n-d}G_{d+i}(I)S_{[d+i+1,j]},
\end{eqnarray}
and that the sets in this union are pairwise disjoint.

Indeed, let $v\in G(I_{[d+j]})$. Then there exist $u\in G(I)$ and $w\in S_j$ such that $v=uw$. Let $i=m(u)$ and suppose that there exists $j<i$ such that $x_j$ divides $w$. Then $v=u'w'$ with $u'=(u/x_i)x_j$ and $w'=(w/x_j)x_i$. Since $I$ is strongly stable, it follows that $u'\in G(I)$. Thus, after finitely modifications like this, we arrive at a presentation $v=u''v''$ with $m(u'')< \min(w'')$. Therefore, if $m(u'')=d+i$,  then  $v\in G_{d+i+1}(I) S_{[d+i+1,j]}$. This shows that $G(I_{[d+j]})$ is the union of the sets $G_{d+i}(I)S_{[d+i+1,j]}$.

It remains to be shown that this union is disjoint. To see this, assume that $$G_{d+i}(I)S_{[d+i+1,j]}\sect G_{d+i'}(I)S_{[d+i'+1,j]}\neq \emptyset,$$ and let $v$ be a monomial in this intersection. Then $v=uw=u'w'$ with $d+i=m(u)<\min(w)$ and $d+i'=m(u')<\min(w')$.

Write  $v=x_{k_1}x_{k_2}\cdots x_{k_{d+j}}$ with $k_1<k_2<\cdots <k_{d+j}$. Then if follows that $d+i=k_d=d+i'$.  Thus, if $i\neq i'$,  then the above intersection is the empty set.

Now it follows from (\ref{disjoint}) that $\mu_{d+j}(I)=\sum_{i=0}^{n-d}|G_{d+i}(I)S_{[d+i+1,j]}|$, and since $$|G_{d+i}(I)S_{[d+i+1,j]}|={n-d-i\choose j}m_{d+i}(I),$$ the desired result follows.
\end{proof}

Lemma~\ref{generators} yields the following polynomial identity

\begin{eqnarray}
\label{muidentity}
\sum_{j=0}^{n-d}\mu_{d+j}(I)t^j=\sum_{j=0}^{n-d}m_{j+d}(I)(1+t)^{n-d-j}.
\end{eqnarray}

\medskip
Now let $\MC$ be a chordal $d$-uniform clutter on the vertex set $[n]$ whose $\lambda$-sequence we want to determine. As explained above, we may assume that $I(\bar{\MC})$ is squarefree strongly stable. Let $\Delta=\Delta(\MC)$ be the clique complex of $\MC$. Then $I_\Delta=I(\bar{\MC})$. Let $\textbf{f}= \left( f_i \right)$ be the $\textbf{f}$-vector of $\Delta$. Together with (\ref{muidentity}) it follows that
\begin{equation}
\label{fm}
\begin{split}
\sum_{j=0}^{n-d}f_{d-1+j}t^j & = \sum_{j=0}^{n-d}{n\choose d+j}t^j-\sum_{j=0}^{n-d}\mu_{d+j}(I(\bar{C}))t^j\\ 
&= \sum_{j=0}^{n-d}{n\choose d+j}t^j- \sum_{j=0}^{n-d}m_{d+j}(1+t)^{n-d-j},
\end{split}
\end{equation}
where by definition,  $m_{d+j}=m_{d+j}(I(\bar{\MC}))$ for all $j$.

On the other hand, Proposition~\ref{fvector} gives us
\begin{align}
\label{prop}
\sum_{j=0}^{n-d}f_{d-1+j}t^j=t^{-1}\sum_{j=1}^{n-d+1}\lambda_j \left((1+t)^j-1 \right).
\end{align}

Combining (\ref{fm}) and (\ref{prop}) we obtain
\begin{align}
\label{mulambda}
\sum_{j=0}^{n-d}m_{j+d}t(1+t)^{n-d-j}=\sum_{j=0}^{n-d}{n\choose d+j}t^{j+1}- \sum_{j=1}^{n-d+1}\lambda_j \left( (1+t)^j-1 \right).
\end{align}

Now we apply to (\ref{mulambda}) the substitution $t\mapsto s-1$ and obtain
\begin{equation} \label{subst}
\begin{split}
\sum_{j=0}^{n-d+1}(m_{n-j+1}-m_{n-j})s^j& = \sum_{j=0}^{n-d}m_{d+j}(s-1)s^{n-d-j}\\ 
&= \sum_{j=0}^{n-d}{n\choose d+j}(s-1)^{j+1}- \sum_{j=1}^{n-d+1}\lambda_j(s^j-1),
\end{split}
\end{equation}
where  $m_{d-1}=m_{n+1}=0$.

\medskip
We define the integers $\alpha_j$ by the identity
\[
\sum_{j=0}^{n-d+1}\alpha_js^j= \sum_{j=0}^{n-d}{n\choose d+j}(s-1)^{j+1}.
\]
More explicitly,
\begin{equation} \label{Def of alpha_i}
\alpha_k = \sum_{j=k}^{n-d+1} \left( -1 \right)^{j-k} {n \choose d+j-1} {j \choose k}.
\end{equation}

Then (\ref{subst}) implies that
\begin{align}
\label{first}
-m_n= \alpha_0+\sum_{j=1}^{n-d+1}\lambda_j,
\end{align}
and
\begin{align}
\label{second}
m_{n-j+1}-m_{n-j}=\alpha_j-\lambda_j\quad \text{for}\quad j=1,\ldots, n-d+1.
\end{align}
Adding up the equation in (\ref{first}) and the equations in (\ref{second}) we obtain
\begin{align}
\label{limi}
m_{n-j}=-\sum_{i=0}^j\alpha_i-\sum_{i=j+1}^{n-d+1}\lambda_i \quad \text{for}\quad j=0,\ldots,n-d.
\end{align}
Reversing these equations we finally find that
\begin{equation} \label{mili}
\begin{split}
\lambda_{n-d+1}&= -\sum_{i=0}^{n-d}\alpha_i-m_d,\quad  \text{and}\\ 
\lambda_{n-d-i}&= \alpha_{n-d-i} + m_{d+i}-m_{d+i+1} \quad \text{for} \quad i=0,\ldots,n-d-1.
\end{split}
\end{equation}

\medskip
Recall, the $i$th Macaulay representation of a positive  integer $a$ is the unique presentation of $a$ as
\[
a=\binom{a(i)}{i}+ \binom{a(i-1)}{i-1}+\cdots +\binom{a(j)}{j} \ \ \mbox{with } \ a(i)>a(i-1)>\cdots>a(j)\geq 1.
\]
One  defines
\[
a^{\langle i \rangle}=\binom{a(i)+1}{i+1}+ \binom{a(i-1)+1}{i}+\cdots +\binom{a(j)+1}{j+1}.
\]
\medskip
A sequence  $l_0,l_1,\ldots,l_m$ of integers is called an   {\em $M$-sequence},   if $l_0=1$ and  $l_{i+1}\leq l_i^{\langle i \rangle}$ for  $i= 1,\ldots,m-1$. By a famous theorem of Macaulay (see for example   \cite[Theorem~4.2.10]{BH}), Hilbert functions of standard graded $K$-algebras are nothing but $M$-sequences.

\medskip
Now we are ready to formulate the main result of this section which characterizes all possible $\lambda$ sequences of chordal clutters.

\begin{thm}
\label{possiblelambda}
Given a sequence of integers $l_0,l_1,\ldots,l_{n-d}$ with $n>d>0$, let
\begin{align*}
\lambda_{n-d-i} & = \alpha_{n-d-i}+l_{i}-l_{i+1},
\end{align*}
for $i=0,\ldots,n-d-1$, where $\alpha_i$ is defined as in (\ref{Def of alpha_i}). Then the following conditions are equivalent:
\begin{enumerate}
\item[(a)] $\lambda_1,\lambda_2,\ldots,\lambda_{n-d}$ is the $\lambda$-sequence of a chordal $d$-uniform clutter $\MC$ on the vertex set $[n]$ with $\mathcal{C} \neq \mathcal{C}_{n,d}$.
\item[(b)] $l_0,l_1,\ldots,l_{n-d}$ is an $M$-sequence with $l_1\leq d$.
\end{enumerate}
\end{thm}

\begin{proof}
As explained above, we may restrict ourselves to  $d$-uniform clutters whose circuit ideal is strongly stable. The condition $\MC\neq \C_{n,d}$ makes sure that $I(\bar{\MC})\neq 0$. We set $l_j=m_{d+j} \left( I \left( \bar{\mathcal{C}} \right) \right)$, for $j=0,\ldots,n-d$. By Murai \cite[Proposition 3.8]{Mu}, the possible sequences $l_0,l_1,\ldots, l_{n-d}$ are exactly those given in (b).  This fact together with (\ref{mili}) yields the desired conclusion.
\end{proof}

\begin{rem}
Since all $m_{n-j}\geq 0$, it follows from (\ref{limi}) that $-\sum_{i=0}^j\alpha_i \geq-\sum_{i=0}^j\alpha_i -\sum_{i=j+1}^{n-d+1}\lambda_i\geq 0$, for all $j=0,\ldots,n-d$. On the other hand, it follows from the definition of the $\alpha_j$ that they are alternating sums of binomial expressions. Therefore, from a numerical point of view it is not obvious why all the partial sums $-\sum_{i=0}^j\alpha_i $ are non-negative. Here we will give direct argument for this fact: for simplicity we denote for $j=0,\ldots,n-d+1$ by $\sigma_j$ the partial sum  $-\sum_{i=0}^j\alpha_i.$ Then
\begin{eqnarray*}
\sum_{j=0}^{n-d+1}\sigma_js^j&+& \sum_{j> n-d+1}\sigma_{n-d+1}s^j =(1+s+s^2+\cdots)(-\sum_{j=0}^{n-d+1}\alpha_js^j)\\
&=&(-\sum_{j=0}^{n-d+1}\alpha_js^j)/(1-s)=(\sum_{j=0}^{n-d+1}\alpha_js^j)/(s-1)= \sum_{j=0}^{n-d}{n\choose d+j}(s-1)^j.
\end{eqnarray*}
This shows that $\sigma_{n-d+1}=0$, and it remains to be shown that  the polynomial $$p_{n,d}=\sum_{j=0}^{n-d}{n\choose d+j}(s-1)^j,\quad \text{with} \quad n\geq d\geq 0,$$ expanded with respect to the powers  of $s$, has non-negative coefficients.

If $d=0$, then $p_{n,0}=\sum_{j=0}^{n}{n\choose j}(s-1)^j = \left((s-1)+1 \right)^n=s^n$, and the result follows in this case. Now we assume that $n\geq d>0$. Since $p_{n+1,d}=p_{n,d}+p_{n,d-1}$, induction on $n+d$ completes the proof.

\medskip
We also see that $p_{n,d}$ is a polynomial of degree $n-d$ with leading coefficient 1. Hence,  $-\sum_{i=0}^{n-d}\alpha_i=1$, and consequently, $\lambda_{n-d+1}= 1- m_d \left( I \left( \bar{\mathcal{C}} \right) \right)$. Therefore, $\lambda_{n-d+1} = 0$ if and only if $\mathcal{C} \neq \mathcal{C}_{n,d}$.
\end{rem}

The integers  $\alpha_i$ are functions of $n$ and $d$. To express this fact we will write in the following statement $\alpha_i(n,d)$ for $\alpha_i$.

 \begin{cor}
 \label{bound}
 Let $\mathfrak{C}_d \left( \left[ n \right] \right)$ be the set of all $d$-uniform chordal clutters on the vertex set $[n]$ and let $i\in [n-d]$. Then the following hold:
 \begin{enumerate}
 \item[(a)] $\max\{\lambda_{i}(\MC)\colon \; \MC\in \mathfrak{C}_d \left( \left[ n \right] \right), \; \mathcal{C} \neq \mathcal{C}_{n,d} \}= \alpha_{i}(n,d)+{n-1-i\choose d-1}$;
 \item[(b)] if $\lambda_1,\ldots, \lambda_{n-d}$ is given by the sequence $l_0,\ldots, l_{n-d}$ as in Theorem~\ref{possiblelambda}, then   $\lambda_{i}= \alpha_{i}(n,d)+{n-1-i \choose d-1}$, if and only if   $l_{n-d-j}= {n-1-j\choose d-1}$ for all  $j\geq i$ and  $l_{n-d-j}= 0$ for all $j<i$.
\end{enumerate}
 \end{cor}

 \begin{proof}
 By Macaulay's theorem on Hilbert functions (see for example \cite[Theorem~4.2.10]{BH} and \cite[Lemma~3.1]{Mu}), the  $M$-sequences as in Theorem~\ref{possiblelambda}(b) are precisely the Hilbert functions of rings of the form $R=K[x_1,\ldots,x_d]/I$, where $I$ is a graded ideal with
$$(x_1,x_2,\ldots,x_d)^{n-d+1}\subset I.$$
Since $\dim_KR_i\leq {d+i-1\choose d-1}$ for all $i$, we see that for $M$-sequence $l_0,l_1,\ldots,l_{n-d}$ with $l_1\leq d$ we have that $l_{n-d-i}\leq {n-1-i\choose d-1}$ for $i=1,\ldots, n-d$. Thus Theorem~\ref{possiblelambda} implies  that
$$\max\{\lambda_{i}(\MC)\colon \; \MC\in \mathfrak{C}_d \left( \left[ n \right] \right), \; \mathcal{C} \neq \mathcal{C}_{n,d} \}\leq  \alpha_{i}(n,d)+{n-1-i\choose d-1}.$$

Note that $\lambda_i(\MC)= \alpha_{i}(n,d)+{n-1-i\choose d-1}$, if and only if $l_{n-d-i}={n-1-i\choose d-1}$ and $l_{n-d-i+1}=0$. This in turn is the case if and only if $l_0+l_1t+\cdots +l_{n-d}t^{n-d}$ is the Hilbert series of the ring $R=K[x_1,\ldots,x_d]/(x_1,\ldots,x_d)^{n-d+1}$. This, indeed completes the proof of both (a) and  (b).
 \end{proof}


Corollary~\ref{bound} together with Theorem~\ref{possiblelambda} immediately implies

\begin{cor}
\label{maximalvalue}
Let $i\in [n-d]$, and suppose that  $\lambda_i(\MC)$ is  maximal among all $\lambda_i(\MC')$ with  $\MC_{n,d} \neq \MC'\in \mathfrak{C}_d \left( \left[ n \right] \right)$. Then
\begin{equation*}
\lambda_j(\MC)	=
\begin{cases}
\alpha_j(n,d) - {n-1-j \choose d-2}, &  \text{if } j> i, \\ 
\alpha_j(n,d)+{n-1-j\choose d-1}, &\text{if } j=i,\\
\alpha_j(n,d),  & \text{if } j<i.
\end{cases}
\end{equation*}
\end{cor}

\begin{rem}
Among the clutters $ \MC \in   \mathfrak{C}_d \left( \left[ n \right] \right)$, for which $\MC \neq \MC_{n,d}$ and  $\lambda_i(\MC)$ is maximal, there is a unique clutter for which  $I(\bar{\MC})$ is squarefree strongly stable. Indeed, let $\MC$ be such a clutter. It follows from the proof of Corollary~\ref{bound} that $m_{n-i}(I(\bar{\MC}))$ takes the maximal value and $m_{n-i+1}(I(\bar{\MC}))=\cdots = m_{n}(I(\bar{\MC}))=0$. The only  clutter $\MC\in   \mathfrak{C}_d \left( \left[ n \right] \right)$ having this property and  for which $I(\bar{\MC})$ is strongly stable is the clutter  $$\MC=\{F\colon \; F\subset [n], |F|=d, F\not\subset [n-i]\}.$$
\end{rem}

\medskip
The above results give no information about the $\lambda$-sequence of a complete clutter. In that case however we have the following very explicit result.

\begin{prop}
	\label{complete}
For  the complete $d$-uniform clutter on the vertex set $[n]$, we have
	\begin{equation*}
		\lambda_i(\MC_{n,d})=
		\begin{cases}
			{n-1-i \choose d-2},  &  \text{for } i=1,\ldots, n-d+1, \\
			0, &\text{otherwise}.
		\end{cases}
	\end{equation*}
\end{prop}

\begin{proof}
Let $T=\{e\subset [n-1]\colon \quad |e|=d-1,\; 1 \in e\}$. We define a total order $\prec$ on $T$  as follows: let $e=\{1, i_1, i_2, \ldots, i_{d-2}\}$ and $e'=\{1, j_1, j_2, \ldots, j_{d-2}\}$ where $1<i_1< i_2< \cdots< i_{d-2}$ and $1< j_1< j_2< \cdots< j_{d-2}$. Then $e \prec e'$ if and only if there exists an integer $t$ such that $i_1=j_1, \ldots, i_{t-1}=j_{t-1}$ and $i_t<j_t$.

We write $T=\{e_1, \ldots, e_m\}$ with $e_1\prec e_2\prec \cdots\prec e_m$, where $m={n-2 \choose d-2}$ and set $\mathcal{C}_0:=\mathcal{C}_{n,d}$ and $\mathcal{C}_i := \mathcal{C}_{i-1} \setminus e_i$, for all $i \geq 1$.
We show that $e_{i}\in \mathrm{Simp} \left( \mathcal{C}_{i-1} \right)$, for all $1 \leq i \leq m$.

Let $i=1$. Then $e_1=\{1,2, \ldots, d-1\}$, and since $\mathcal{C}_{n,d}$ is $d$-complete, we have $\mathrm{N}_\C \left[ e_1 \right]=[n]$ which is a clique of $\mathcal{C}_{n,d}= \C_0$ and so $e_{1} \in \mathrm{Simp}(\mathcal{C}_{0})$. Suppose that $i>1$ and that $e_i=\{1, i_1, \ldots, i_{d-2}\}$.

First we show  that $\mathrm{N}_{\mathcal{C}_{i-1}} \left[ e_i \right] = e_i \cup \{i_{d-2}+1, \ldots, n\}$.
Let $j\in \{i_{d-2}+1, \ldots, n\}$. To prove that $j \in \mathrm{N}_{\mathcal{C}_{i-1}} \left[ e_i \right]$, we show that $e_i \cup \{ j \} \in \mathcal{C}_{i-1}$. Since $e_i \cup \{ j \} \in \mathcal{C}_{0}$, it is enough to prove that $e_k \nsubseteq e_i \cup \{ j \}$, for all $1\leq k\leq i-1$.

Suppose that there exists $1 \leq k \leq m$, $k\neq i$, such that $e_k \subset e_i\cup \{ j \}$. So $e_k= \{ j \} \cup e_i \setminus \{ i_s \}$, for some $1 \leq s \leq d-2$. Since $j>i_l$ for all $1\leq l\leq d-2$, we have $e_k > e_i$. So $k>i$ which implies that $e_k \nsubseteq e_i \cup \{ j \}$, for all $1 \leq k \leq i-1$.
Thus $\{ i_{d-2}+1, \ldots, n\} \subset \mathrm{N}_{\mathcal{C}_{i-1}} \left[ e_i \right]$.

Conversely,  suppose that $j \in \mathrm{N}_{\mathcal{C}_{i-1}} \left[ e_i \right] \setminus e_i$. Then $e_i \cup \{ j \} \in \mathcal{C}_{i-1}$.
Thus $e_k \nsubseteq e_i\cup \{ j \}$, for all $1\leq k\leq i-1$. In particular $e_k \neq  \{ j \} \cup e_i \setminus \{ i_l \}$, for any $1\leq l\leq d-2$. If $j=n$, then $j \in \{ i_{d-2}+1, \ldots, n \}$. Suppose that $j \neq n$. Then $\{ j \} \cup e_i \setminus \{ i_l \}\in T$. Then there exists $i+1 \leq k_l \leq m$ such that $e_{k_l} = \{ j \} \cup e_i \setminus \{i_l\}$, for all $1 \leq l \leq d-2$.  By the way of ordering of elements of $T$, $k_{d-2} > i$ implies that  $e_{k_{d-2}} > e_i$ and so $j > i_{d-2}$. So  $\mathrm{N}_{\mathcal{C}_{i-1}} \left[ e_i \right] \setminus e_i \subseteq \{ i_{d-2}+1, \ldots, n \}$.
Hence $\mathrm{N}_{\mathcal{C}_{i-1}} \left[ e_i \right] = \{i_{d-2}+1, \ldots, n\}\cup  e_i$.

Now we prove that $\mathrm{N}_{\mathcal{C}_{i-1}} \left[ e_i \right]$ is a clique in $\mathcal{C}_{i-1}$.
Let $F \subseteq \mathrm{N}_{\mathcal{C}_{i-1}} \left[ e_i \right]$ with $|F|=d$. Assume that $e_k \subset F$ and $e_k \in T$. Then either $k=i$, or there exists at least one $j \in  \{ i_{d-2}+1, \ldots, n-1 \}$ such that $j \in e_k$. Since $j>i_l$ for all $i_l\in e_i$, it follows that $e_k > e_i$ and so $k > i$. Therefore $e_k \not\subset F$, for all $1 \leq k \leq i-1$. It follows that $F \in \mathcal{C}_{i-1}$.

One may easily check that $\mathcal{C}_{r-1} \setminus e_r =\mathcal{C}_{n-1,d}$. The induction on $[n]$, now shows that $\mathcal{C}_{n,d}$ is a chordal clutter.

Let $\textbf{e} = e_1, \ldots, e_r$ be the simplicial sequence as above, $\mathcal{N}' = \{N'_1, \ldots, N'_s\}$ be the multiset of $\C':= \C_{n-1,d}$ and $\lambda'_j =|\{N'_i \colon \quad N'_i =j\}|$. One may verify that $\lambda_{n-d+1} = 1$ (because $e_1$ is the only simplicial submaximal circuit with $\lvert \mathrm{N}_\C \left( e_1 \right) \rvert = n-d+1$) and
\[
\lambda_j = \lambda'_j + \left| \left\{ i \colon \; \lvert \mathrm{N}_{\C_{i-1}} \left(e_i \right) \rvert =j \right\} \right|,
\]
for $1 \leq j \leq n-d$. We may use induction (on the number of vertices of $\C$) to obtain, $\lambda'_j =  {{n-1-j-1} \choose d-2}$. Note that $\lvert \left\{ i \colon \; \lvert \mathrm{N}_{\C_{i-1}} \left(e_i \right) \rvert =j \right\} \rvert = {n-j-2 \choose d-3}$. Therefore,
\[
\lambda_j = {n-1-j-1 \choose d-2} + {n-j-2 \choose d-3} = {n-j-1 \choose d-2}.
\]		
\end{proof}
One would expect that if $\MC$ is a $d$-uniform chordal clutter on the vertex set $[n]$, then $\lambda_i(\MC)\leq \lambda_i(\MC_{n,d})$, where  $\lambda_i(\MC_{n,d})= {n-1-i \choose d-2}$, by Proposition~\ref{complete}. However this is not the case, as the following example shows: let $\MC=\{123,124,134\}$. Then $\MC$ is chordal with $\lambda(\MC)=3,0,\ldots$. However, $\lambda(\MC_{4,3})=2,1,0,\ldots$.

Let $\MC$ be a $d$-uniform clutter on the vertex set $[n]$. We call $\MC$ {\em co-chordal}, if there exists a simplicial sequence  $e_1,\ldots,e_r$ in $\mathcal{C}_{n,d}$ such that
\[
\MC=(\MC_{n,d})_{e_1e_2\ldots e_r}.
\]
At present we do not know whether any co-chordal clutter is chordal. The previous example shows a chordal clutter need not to be co-chordal.

\begin{cor}
	\label{co-chordal}
	Let $\MC$ be a $d$-uniform clutter on the vertex set $[n]$ which is both chordal and co-chordal. Then for all $i$,
	\[
	\lambda_i(\MC)\leq {n-1-i \choose d-2}.
	\]
\end{cor}

\begin{proof}
	There exist simplicial sequences $e_1,\ldots,e_r$ in $\mathcal{C}_{n,d}$ and $e_{r+1},\ldots,e_s$ in $\mathcal{C}$ such that $\MC=(\MC_{n,d})_{e_1e_2\cdots e_r}$ and $\MC_{e_{r+1}\cdots e_s}=\emptyset$.
Thus $e_1,e_2,\ldots,e_s$ is a simplicial order for $\MC_{n,d}$. It follows that the multiset of $\MC$ is a sub-multiset of that of $\MC_{n,d}$. This yields the desired conclusion.
\end{proof}


\begin{thebibliography}{99}
	
\bibitem{Nevo}
K.~A.~Adiprasoto, E.~Nevo and J.~A.~Samperhigher, \textit{Higher chordality $\rm I$: From graphs to complexes}, To appear in Proc. Amer. Math. Soc., (2015). \href{http://arxiv.org/abs/1503.05620}{\texttt{arXiv:1503.05620}}

\bibitem{BS}
D.~Bayer and M.~Stillman, \textit{A criterion or detecting $m$-regularity}, Invent.~Math. \textbf{87}, pp. 1--11 (1987).
	
\bibitem{BYZ} M.~Bigdeli, A.~A~Yazdan Pour and R.~Zaare-Nahandi, \textit{Stability of Betti numbers under reduction processes: towards chordality of clutters}, preprint, (2015). \href{http://arxiv.org/abs/1508.03799}{\texttt{arXiv:1508.03799}}


\bibitem{BH}
W.~Bruns and J.~Herzog, \textit{Cohen-Macaulay rings}, Revised Edition, Cambridge University Press, Cambridge, (1996).


\bibitem{CHH}
A.~Conca, J.~Herzog and T.~Hibi, \textit{Rigid resolutions and big Betti numbers}, Comment.~Math.~Helv. \textbf{79}, pp. 826--839 (2004).


\bibitem{ConnonFaridi}
E.~Connon and S.~Faridi, \textit{Chorded complexes and a necessary condition for a monomial ideal to have a linear resolution}, J.~Combin.~Theory~Ser.~A,  \textbf {120}, pp. 1714--1731 (2013).

\bibitem{Dirac}
G.~A.~Dirac, \textit{On rigid circuit graphs}, Abh. Math. Sem. Univ. Hamburg \textbf{38}, pp. 71--76, (1961).


\bibitem{EK}
S.~Eliahou and M.~Kervaire, \textit{Minimal resolutions of some monomial ideals}, J. Algebra \textbf{129}, pp. 1--25
(1990).

\bibitem{ERT}
D.~Eisenbud, A.~Reeves and B.~Totaro, \textit{Initial ideals, Veronese subrings, and rates of algebras}, Adv.~Math. \textbf{109}, pp. 168--187 (1994).

\bibitem{Emtander}
E.~Emtander, \textit{A class of hypergraphs that generalizes chordal graphs}, Math.~Scand. \textbf{106},  pp. 50--66 (2010).

\bibitem{Fr}
R.~Fr\"oberg,  \textit{On Stanley--Reisner rings},
in: Topics in algebra, {Banach Center Publications} \textbf{26}, pp. 57--70 (1990).

	
	
\bibitem{HH}
J.~Herzog and T.~Hibi, \textit{Monomial Ideals}, in: GTM 260, Springer, London, (2010).
	
\bibitem{mar2}
M.~Morales, A.~A.~Yazdan Pour and R.~Zaare-Nahandi, \textit{Regularity and Free Resolution of Ideals which are Minimal to $d$-linearity}, To appear in Math. Scand. (2016).
	
\bibitem{MNYZ}
M.~Morales, A.~Nasrollah Nejad, A.~A.~Yazdan Pour and R.~Zaare-Nahandi, \textit{Monomial ideals with $3$-linear resolutions}, Annales de la Facult\'e des Sciences de Toulouse, S\'er. \textbf{6}, 23: (4),  pp. 877--891, (2014). 
	
\bibitem{Mu}
S.~Murai, \textit{Hilbert functions of $d$-regular ideals}, J.~Algebra \textbf{317}, pp. 658--690 (2007).

\bibitem{XY}
A.~Nikseresht and R.~Zaare-Nahandi, \textit{On Generalizations of cycles and chordality to hypergraphs}, preprint, (2016). \href{http://arxiv.org/abs/1601.03207}{\texttt{arXiv:1601.03207}}.


\bibitem{Woodroofe}
R.~Woodroofe, \textit{Chordal and sequentially Cohen-Macaulay clutters}, Electron.~J.~Combin. {\textbf 18}, no. 1, Paper 208, 20 pages, (2011). 
	
\bibitem{code}
A Program for Detecting Chordality. available at: \href{http://iasbs.ac.ir/~yazdan/chordality.html}{\texttt{www.iasbs.ac.ir/$\sim$yazdan/chordality.html}}

\bibitem{code2}
A Program for computing numerical data of chordal clutters. available at: \href{http://iasbs.ac.ir/~yazdan/numericaldata.html}{\texttt{www.iasbs.ac.ir/$\sim$yazdan/numericaldata.html}}
\end{thebibliography}
\end{document}